\documentclass[graybox]{svmult}

\smartqed
\usepackage{mathptmx}       
\usepackage{helvet}         
\usepackage{courier}        
\usepackage{type1cm}        
\usepackage{graphicx}       
\usepackage{subfigure}

\usepackage{array,colortbl}
\usepackage{amsmath,amsfonts,amssymb,bm} 
\DeclareFontFamily{U}{mathx}{\hyphenchar\font45}
\DeclareFontShape{U}{mathx}{m}{n}{
      <5> <6> <7> <8> <9> <10>
      <10.95> <12> <14.4> <17.28> <20.74> <24.88>
      mathx10
      }{}
\DeclareSymbolFont{mathx}{U}{mathx}{m}{n}
\DeclareFontSubstitution{U}{mathx}{m}{n}
\DeclareMathAccent{\widecheck}      {0}{mathx}{"71}

\renewcommand{\email}[1]{\emailname: #1} 

\usepackage{mathptmx}       
\usepackage{helvet}         
\usepackage{courier}        

\usepackage{makeidx}         
\usepackage{graphicx}        
\usepackage[bottom]{footmisc}

\usepackage{latexsym}
\usepackage{amsmath}
\usepackage{amsfonts}
\usepackage{amssymb}
\usepackage{bm}

\usepackage{url}
\usepackage{algorithm}
\usepackage{algorithmic}
\usepackage[misc,geometry]{ifsym}

\renewenvironment{proof}{\noindent{\itshape Proof.}}{\smartqed\qed}

\spdefaulttheorem{assumption}{Assumption}{\upshape \bfseries}{\itshape}
\spdefaulttheorem{algo}{Algorithm}{\upshape \bfseries}{\itshape}

\newcommand{\rmK}{{\mathrm{K}}}

\newcommand{\bbC}{{\mathbb{C}}}
\newcommand{\bbD}{{\mathbb{D}}}
\newcommand{\bbE}{{\mathbb{E}}}

\newcommand{\bbN}{{\mathbb{N}}}

\newcommand{\bbP}{{\mathbb{P}}}

\newcommand{\bbT}{{\mathbb{T}}}

\newcommand{\bbV}{{\mathbb{V}}}

\newcommand{\bbX}{{\mathbb{X}}}

\newcommand{\N}{{\mathbb{N}}} 
\newcommand{\R}{{\mathbb{R}}} 
\newcommand{\EE}{{\mathbb{E}}}

\DeclareSymbolFont{bbold}{U}{bbold}{m}{n}
\DeclareSymbolFontAlphabet{\mathbbold}{bbold}

\newcommand{\calA}{{\mathcal{A}}}
\newcommand{\calB}{{\mathcal{B}}}

\newcommand{\calO}{{\mathcal{O}}}

\newcommand{\calT}{{\mathcal{T}}}
\newcommand{\calU}{{\mathcal{U}}}

\newcommand{\ceil}[1]{\left\lceil #1 \right\rceil}    
\DeclareMathOperator{\cov}{Cov}
\DeclareMathOperator{\var}{Var}

\newcommand*{\lrscript}[5]{{\vphantom{#1}}_{#2}^{#3}{#1}_{#4}^{#5}}
\newcommand*{\dualpair}[4]{\ensuremath{\lrscript{\langle}{#1}{}{}{} #3 ,
		#4 \rangle_{#2}}}

\DeclareSymbolFont{bbold}{U}{bbold}{m}{n}
\DeclareSymbolFontAlphabet{\mathbold}{bold}

\usepackage{microtype} 

\usepackage[colorlinks=true,linkcolor=black,citecolor=black,urlcolor=black]{hyperref}
\usepackage{bookmark}
\makeatletter 
  \providecommand*{\toclevel@author}{999}
  \providecommand*{\toclevel@title}{0}
\makeatother

\begin{document}

\title*{A Multilevel Monte Carlo Algorithm for Parabolic Advection-Diffusion Problems with Discontinuous Coefficients}
\titlerunning{MLMC for Discontinuous Advection-Diffusion Problems}

\author{Andreas Stein \and Andrea Barth}
\institute{
Andreas Stein \and Andrea Barth 
\at SimTech, University of Stuttgart, Allmandring 5b, 70569 Stuttgart, Germany  
\email{andreas.stein@mathematik.uni-stuttgart.de}, \email{andrea.barth@mathematik.uni-stuttgart.de} 
}
\maketitle

\abstract{The Richards' equation is a model for flow of water in unsaturated soils. The coefficients of this (nonlinear) partial differential equation describe the permeability of the medium. Insufficient or uncertain measurements are commonly modeled by random coefficients. For flows in heterogeneous\textbackslash fractured\textbackslash porous media, the coefficients are modeled as discontinuous random fields, where the interfaces along the stochastic discontinuities represent transitions in the media. More precisely, the random coefficient is given by the sum of a (continuous) Gaussian random field and a (discontinuous) jump part. In this work moments of the solution to the random partial differential equation are calculated using a path-wise numerical approximation combined with multilevel Monte Carlo sampling. The discontinuities dictate the spatial discretization, which leads to a stochastic grid. Hence, the refinement parameter and problem-dependent constants in the error analysis are random variables and we derive (optimal) a-priori convergence rates in a mean-square sense.}

\keywords{Multilevel Monte Carlo method, flow in heterogeneous media, fractured media, porous media, jump-diffusion coefficient, non-continuous random fields, parabolic equation, advection-diffusion equation}

\section{Introduction}

We consider a linear (diffusion-dominated) advection-diffusion equation with random L\'evy fields as coefficients. Adopting the term from stochastic analysis, by a L\'evy field we mean a random field which is built from a (continuous) Gaussian random field and a (discontinuous) jump part (following a certain jump measure). In the last decade various ways to approximate the distribution or moments of the solution to a random equation were introduced. Next to classical Monte Carlo methods, their multilevel variants and further variance reduction techniques have been applied. Due to their low regularity constraints, multilevel Monte Carlo techniques have been successfully applied to various problems, for instance in the context of elliptic random PDEs in \cite{ABS13, BSZ11, CGST11, LWZ16, GSTU13, BS18b} to just name a few. These sampling approaches differ fundamentally from Polynomial-Chaos-based methods. The latter suffer from high regularity assumptions. While in the case of continuous fields these algorithms can outperform sampling strategies, approaches -- like stochastic Galerkin methods -- are less promising in our discontinuous setting. In fact, it is even an open problem to define them for L\'evy fields. While Richards' equation formulated as a deterministic interface problem was considered in numerous publications (see~\cite{DHK91, FO17} and the references therein), there is up-to-date no stochastic formulation.

After introducing the necessary basic notation, in this paper we show in Section~\ref{astein_sec:prelim} existence and uniqueness of a path-wise weak solution to the random advection-diffusion equation and prove an energy estimate which allows for a moment estimate. Next to space- and time-discretizations, the L\'evy field has to be approximated, resulting in an approximated path-wise weak solution. In Section~\ref{astein_sec:para_disc} we show convergence of this approximated path-wise weak solution, before we introduce a sample-adapted (path-wise) Galerkin approximation. Only if the discretization is adapted to the random discontinuities can we expect full convergence rates. 
As the main result of this article, we prove the error estimate of the spatial discretization in the $L^2$-norm.
To this end, we utilize the corresponding results with respect to the $H^1$-norm from \cite{BS18c} and consider the parabolic dual problem. 
Finally, we combine the sample-adapted spatial discretization with a suitable time stepping method to obtain a fully discrete path-wise scheme. 
The path-wise approximations are used in Section~\ref{astein_sec:mlmc} to estimate quantities of interest using a (coupled) multilevel Monte Carlo method. Naturally, the optimal sample numbers on each level depend on the sample-dependent convergence rate. The term \textit{coupled} refers to a simplified version of \textit{Multifidelity Monte Carlo} sampling (see~\cite{GPW16}) that reuses samples across levels and is preferred when sampling from a certain distribution is computationally expensive. In Section~\ref{astein_sec:num}, a numerical example confirms our theoretical results from Section~\ref{astein_sec:para_disc} and shows that the sample-adapted strategy vastly outperforms a multilevel Monte Carlo estimator with a standard Finite Element discretization in space. 
 
\section{Parabolic Problems with Random Discontinuous Coefficients}\label{astein_sec:prelim}
Let $(\Omega,\calA,\bbP)$ be a complete probability space, $\bbT=[0,T]$ be a time interval for some $T>0$ and $\bbD\subset\R^d$, $d\in\{1,2\}$, be a polygonal and convex domain.
We consider the linear, random initial-boundary value problem
\begin{equation}\label{astein_eq:pde}
\begin{split}
\partial_t u(\omega,x,t)+[Lu](\omega,x,t)&=f(\omega,x,t)\quad\text{in $\Omega\times \bbD\times(0,T])$},\\
u(\omega,x,0)&=u_0(\omega,x)\quad\text{in $\Omega\times \bbD\times\{0\}$},\\
u(\omega,x,t)&=0\quad\text{on $\Omega\times\partial \bbD\times\bbT$},
\end{split}
\end{equation}
where  $f:\Omega\times \bbD\times\bbT\to\R$ is a random source function and $u_0:\Omega\times \bbD$ denotes the initial condition of the above PDE.
Furthermore, $L$ is the second order partial differential operator given by 
\begin{equation}\label{astein_eq:L}
\begin{split}
[Lu](\omega,x,t)=-\nabla\cdot \left(a(\omega,x)\nabla u(\omega,x,t)\right)+b(\omega,x)\mathbf 1^T\nabla u(\omega,x,t)\\
\end{split}
\end{equation}
for $(\omega,x,t)\in\Omega\times \bbD\times\bbT$ with $\nabla$ operating on the second argument of $u$.
In Eq.~\eqref{astein_eq:L}, we set $\mathbf 1:=(1,\dots,1)^T\in\R^n$, such that $\mathbf 1^T\nabla u=\sum_{i=1}^n\partial_{x_i} u$, and consider
\begin{itemize}
	\item a stochastic jump-diffusion coefficient $a:\Omega\times \bbD\to\R$ and
	\item a random discontinuous convection term $b:\Omega\times \bbD\to\R$ coupled to $a$.
\end{itemize}

Throughout this article, we denote by $C$ a generic positive constant which may change from one line to the next. 
Whenever helpful, the dependence of $C$ on certain parameters is made explicit. 
To obtain a path-wise variational formulation, we use the standard Sobolev space $H^s(\bbD)$ with norm $\|\cdot\|_{H^s(\bbD)}$ for any $s>0$, see for instance \cite{AF03, DGV12}.  
Since $\bbD$ has a Lipschitz boundary, for $s\in(1/2,3/2)$, the existence of a bounded, linear trace operator $\gamma:H^s( \bbD)\to H^{s-1/2}(\partial \bbD)$  
is ensured by the trace theorem, see~\cite{D96}.
We only consider homogeneous Dirichlet boundary conditions on $\partial \bbD$, hence we may treat $\gamma$ independently of $\omega\in \Omega$
and define the suitable solution space $V$ as 
\begin{equation*}
V:=H_0^1( \bbD)=\{v\in H^1( \bbD)|\;\gamma v\equiv0\}, 
\end{equation*}
equipped with the $H^1( \bbD)$-norm $ \|v\|_V:= \|v\|_{H^1( \bbD)}$. 
With $H:=L^2( \bbD)$, we work on the Gelfand triplet $V\subset H\subset V'=H^{-1}( \bbD)$, where $V'$ denotes the topological dual of $V$, i.e. the space of all bounded, linear functionals on $V$. In the variational version of Problem~\eqref{astein_eq:pde}, $\partial_tu$ denotes the weak time derivative of $u$. Throughout this article, we may as well consider $\partial_t u$ as derivative in a strong sense (also with regard to its approximation at the end of Section~\ref{astein_sec:para_disc}) as we will always assume sufficient temporal regularity.
As the coefficients $a$ and $b$ are random functions, any solution $u$ to Problem~\eqref{astein_eq:pde} is a time-dependent $V$-valued random variable. To investigate the regularity of the solution $u$ with respect to $\bbT$ and the underlying probability measure $\bbP$ on $\Omega$, we need to introduce the corresponding Lebesgue-Bochner spaces.
To this end, let $p\in[1,\infty)$ and$(\bbX,\|\cdot\|_\bbX)$ be an arbitrary Banach space.
For $Y\in\{\bbT, \Omega\}$, the Lebesgue-Bochner space $L^p(Y;\bbX)$ is defined as
\begin{equation*}
L^p(Y;\bbX):=\{\varphi:Y\to\bbX\text{ is strongly measurable and }\|\varphi\|_{L^p(Y;\bbX)}<+\infty\},
\end{equation*}
with the norm
\begin{equation*}
\|\varphi\|_{L^p(Y;\bbX)}:=\begin{cases}
\Big(\int_\bbT\|\varphi(t)\|_\bbX^pdt\Big)^{1/p}\quad\text{for $Y=\bbT$,}\\
\EE(\|\varphi\|^p)^{1/p}=\Big(\int_\Omega\|\varphi(\omega)\|_\bbX^pd\bbP(d\omega)\Big)^{1/p}\quad\text{for $Y=\Omega$.}
\end{cases}.
\end{equation*}

The bilinear form associated to $L$ is introduced to derive a weak formulation of the initial-boundary value problem~\eqref{astein_eq:pde}.
For fixed $\omega\in\Omega$ and $t\in\bbT$, multiplying Eq.~\eqref{astein_eq:pde} with a test function $v\in V$ and integrating by parts yields
\begin{equation}\label{astein_eq:var}
\dualpair{V'}{V}{\partial_t u(\omega,\cdot,t)}{v}+B_\omega(u(\omega,\cdot,t),v)=\dualpair{V'}{V}{f(\omega,\cdot,t)}{v}.
\end{equation}
The bilinear form $B_{\omega}:V\times V\to\R$ is given by  
\begin{align*}
B_\omega(u,v)=\int_ \bbD a(\omega,x)\nabla u(x)\cdot\nabla v(x)+b(\omega,x)\mathbf 1^T\nabla u(x)v(x)dx,
\end{align*}
and $\dualpair{V'}{V}{\cdot}{\cdot}$ denotes the $(V',V)$-duality pairing. 

\begin{definition}\label{astein_def:weak}
	For fixed $\omega\in\Omega$, the \textit{path-wise weak solution} to Problem~\eqref{astein_eq:pde} is a function $u(\omega,\cdot,\cdot)\in L^2(\bbT;V)$ with $\partial_t u(\omega,\cdot,\cdot)\in L^2(\bbT;V')$ such that, for $t\in\bbT$,
	\begin{equation*}
	\dualpair{V'}{V}{\partial_t u(\omega,\cdot,t)}{v}+B_\omega(u(\omega,\cdot,t),v)=\dualpair{V'}{V}{f(\omega,\cdot,t)}{v},\quad \text{ for all}\;v\in V
	\end{equation*}
	and $u(\omega,\cdot,0)=u_0(\omega,\cdot)$. Furthermore, we define the path-wise parabolic norm by
	\begin{equation}\label{astein_eq:e_norm}
	\begin{split}
	\|u(\omega,\cdot,\cdot)\|_{*,t}:&=\Big(\|u(\omega,\cdot,t)\|_H^2+\int_0^t\int_\bbD\nabla u(\omega,x,z)\cdot\nabla u(\omega,x,z)dxdz\Big)^{1/2}\\
	&=\Big(\|u(\omega,\cdot,t)\|_H^2+\|\|\nabla u(\omega,x,z)\|_2\|^2_{L^2([0,t];H)}\Big)^{1/2},
	\end{split}
	\end{equation}
	where $\|\cdot\|_2$ is the Euclidean norm on $\R^d$.
\end{definition}

To represent the (uncertain) permeability in a subsurface flow model, we use the random jump coefficients $a,b$ from the elliptic/parabolic problems in~\cite{BS18b, BS18c}. 
The diffusion coefficient is then given by a (spatial) Gaussian random field with additive discontinuities on random areas of $ \bbD$. 
Its specific structure may be utilized to model the hydraulic conductivity within heterogeneous and/or fractured media and thus $a$ is considered time-independent.
The advection term in this model is driven by the same random field and inherits the same discontinuous structure as the diffusion, hence we consider the coefficient $b$ as a linear mapping of $a$.

\begin{definition}\label{astein_def:a}
	The \textit{jump-diffusion coefficient} $a$ is defined as
	\begin{equation*}
	a:\Omega\times\bbD\to\R_{>0},\quad (\omega,x)\mapsto\overline a(x)+\Phi(W(\omega,x))+P(\omega,x),
	\end{equation*}
	where
	\begin{itemize}
		\item $\overline a\in C^1(\overline\bbD;\R_{\ge0})$ is non-negative, continuous, and bounded.
		\item $\Phi\in C^1(\R;\R_{>0})$ is a continuously differentiable, positive mapping.
		\item $W\in L^2(\Omega;H)$ is a (zero-mean) Gaussian random field associated to a non-negative, symmetric trace class operator $Q:H\to H$. 
		\item $\calT:\Omega\to\calB(\bbD),\;\omega\mapsto\{\calT_1,\dots,\calT_{\tau}\}$ is a random partition of $\bbD$, i.e. the $\calT_i$ are disjoint open subsets of $\bbD$ such that $|\calT_i|>0$ and $\overline\bbD=\bigcup_{i=1}^\tau\overline\calT_i$, and $\calB(\bbD)$ denotes the Borel-$\sigma$-algebra on $\bbD$.
		The number of elements in $\calT$, $\tau$, is a random variable on $(\Omega,\calA,\bbP)$, i.e. $\tau:\Omega\to\N$. 		
		\item $(P_i, i\in\N)$ is a sequence of non-negative random variables on $(\Omega,\calA,\bbP)$ and
		\begin{equation*}
		P:\Omega\times\bbD\to\R_{\ge0},\quad(\omega,x)\mapsto\sum_{i=1}^{\tau(\omega)}\mathbf 1_{\{\calT_i\}}(x)P_i(\omega).
		\end{equation*}
		The sequence $(P_i, i\in\N)$ is independent of $\tau$ (but not necessarily i.i.d.).
	\end{itemize}
	Based on $a$, the \textit{jump-advection coefficient} $b$ is given for $b_1,b_2\in L^\infty(\bbD)$ by
	\begin{equation*}
	b:\Omega\times\bbD\to\R,\quad (\omega,x)\mapsto \min(b_1(x)a(\omega,x), b_2).
	\end{equation*}
\end{definition}

The definition of the random partition $\calT$ above is rather general and does not yet assume any structure on the discontinuities.   
A more specific class of random partitions is considered in our numerical experiment in Section~\ref{astein_sec:num}.
We assumed in Definition~\ref{astein_def:a} that $\tau$ and $P_i$ are independent due to technical reasons, i.e. to control for a possible sampling bias in $P_i$, see \cite[Theorem 3.11]{BS18b}. On a further note, we do not require stochastic independence of $W$ and $P$. 
In general, our aim is to estimate moments of a \textit{quantity of interest} (QoI) $\Psi (\omega):=\psi(u(\omega,\cdot,\cdot))$ of the weak solution, where $\psi:L^2(\bbT; V)\to\R$ is a deterministic functional.
To ensure existence and a certain regularity of $u$, and therefore of $\Psi$, we fix the following set of assumptions.

\begin{assumption}\label{astein_ass:EV}
	~
	\begin{enumerate}
		\item Let $\eta_1\ge\eta_2\ge\dots\ge0$ denote the eigenvalues of $Q$ in descending order and $(e_i,i\in\N)\subset H$ be the corresponding eigenfunctions. 
		The $e_i$ are continuously differentiable on $ \bbD$ and there exist constants $\alpha,\beta,C_e,C_\eta>0$ such that $2\alpha\le \beta$ and for any $i\in\N$
		\begin{equation*}
		\|e_i\|_{L^\infty( \bbD)}\le C_e,\quad\max_{j=1,\dots,d}\|\partial_{x_j} e_i\|_{L^\infty( \bbD)}\le C_ei^\alpha\quad\text{and}\quad\sum_{i=1}^\infty\eta_ii^\beta\leq C_\eta<+\infty.
		\end{equation*} 
		
		\item Furthermore, the mapping $\Phi$ as in Definition~\ref{astein_def:a} and its derivative are bounded by
		\begin{equation*}
		\phi_1\exp(\phi_2 |w|)\ge \Phi(w)\ge\phi_1\exp(-\phi_2 |w|),\quad|\frac{d}{dx}\Phi(w)|\le\phi_3\exp(\phi_4|w|),\quad w\in\R,
		\end{equation*}
		where $\phi_1,\dots,\phi_4>0$ are arbitrary constants.
		\item For some $p>2$, $f,\partial_t f\in L^p(\Omega;L^2(\bbT;H)), u_0\in L^p(\Omega;H^2(\bbD)\cap V)$ and $u_0$ and $f$ are stochastically independent of $\calT$.
		\item The partition elements $\calT_i$ are almost surely polygons with piecewise linear boundary and $\bbE(\tau^n)<+\infty$ for all $n\in\bbN$.
		\item The sequence $(P_i, i\in\N)$ consists of nonnegative and bounded random variables $P_i\in [0,\overline P]$ for some $\overline P>0$.
		\item The functional $\psi$ is Lipschitz continuous on $L^2(\bbT;H)$, i.e. there exists $C_\psi>0$ such that 
		\begin{equation*} 
		|\psi(v)-\psi(w)|\le C_\psi\|v-w\|_{L^2(\bbT;H)}\quad \forall v,u\in L^2(\bbT;H).
		\end{equation*} 
	\end{enumerate}
\end{assumption}  
 
\begin{remark}
The above assumptions are natural and cannot be relaxed significantly to derive the results in Section~\ref{astein_sec:para_disc}.
The condition $2\alpha\le\beta$ implies that $W$ has almost surely Lipschitz continuous paths on $\bbD$, thus $a$ is piecewise Lipschitz continuous. 
This is in turn necessary to derive the error estimates of orders $\calO(\overline h_\ell^\kappa)$ and $\calO(\overline h_\ell^{2\kappa})$ in Theorem~\ref{astein_thm:semi_error_V} and Theorem~\ref{astein_thm:semi_error_H}, respectively, for some $\kappa\in (1/2,1]$ that is independent of $W$.
The parameter $\overline h_\ell$ denotes the Finite Element (FE) refinement and $\kappa$ should only be influenced by the law of the random jump field $P$.
If any of this assumptions were violated, however, $\kappa$ may depend on other parameters of the random PDE.
For instance, if $\beta/2\alpha<\kappa\le 1$, we would only obtain an error of approximate order $\calO(\overline h_\ell^{\beta/2\alpha})$ in Theorem~\ref{astein_thm:semi_error_V}, see \cite{BS18c} for a detailed discussion.
The remaining points in Assumption~\ref{astein_ass:EV} ensure that all estimates hold in the mean-square sense, i.e. the second moments of all estimates exist and can be bounded with respect to $\overline h_\ell$. 
\end{remark}

We have the following estimate on $a$ and its piecewise Lipschitz norm.

\begin{lemma}\label{astein_lem:a}\cite[Lemmas 3.6 and 4.8]{BS18c}
	Let Assumption~\ref{astein_ass:EV} hold and define $a_{-}(\omega):=\text{ess inf}_{x\in \bbD} a(\omega,x)$ and $a_{+}(\omega):=\text{ess sup}_{x\in \bbD} a(\omega,x)$. Then, for any $q\in[1,\infty)$
	\begin{equation*}
	1/a_-,\;a_+,\;\max_{i=1,\dots,\tau}\sum_{j=1}^d\|\partial_{x_j}a\|_{L^\infty( \calT_i)}\in L^q(\Omega;\R).
	\end{equation*}
\end{lemma}

\begin{theorem}\label{astein_thm:exis}
	Under Assumption~\ref{astein_ass:EV} there exists almost surely a unique path-wise weak solution $u(\omega,\cdot,\cdot)\in L^2(\bbT;V)$ to Problem~\eqref{astein_eq:pde} satisfying the estimate
	\begin{equation}\label{astein_eq:e_norm_pw}
	\sup_{t\in\bbT} \|u(\omega,\cdot,\cdot)\|^2_{*,t}\le C/a_-(\omega)\Big(\|u_0(\omega,\cdot)\|^2_H+\|f(\omega,\cdot,\cdot)\|^2_{L^2(T;H)}\Big)<+\infty.
	\end{equation}
	In addition, for any $r\in[1,p)$ (with $p$ as in Ass.~\ref{astein_ass:EV}), $u$ is bounded in expectation by
	\begin{equation}\label{astein_eq:e_norm_exp}
	\bbE\Big(\sup_{t\in\bbT} \|u\|^r_{*,t}\Big)^{1/r}\le C\|1/a_-\|_{L^{\widetilde q}(\Omega;\R)}\left(\|u_0\|_{L^p(\Omega;H)}+\|f\|_{L^p(\Omega;L^2(\bbT;V'))}\right)<+\infty.
	\end{equation}
	with $C=C(r)$ and $\widetilde q:=(1/r-1/p)^{-1}$. Furthermore, it holds $\Psi\in L^r(\Omega;\R)$.
\end{theorem}
\begin{proof}
	The estimates in Ineq.~\eqref{astein_eq:e_norm_pw} and \eqref{astein_eq:e_norm_exp} follow from \cite[Theorem 3.7]{BS18c}. 
	To show that $\Psi\in L^r(\Omega;\R)$, we use Assumption~\ref{astein_ass:EV} to see that $\psi$ fulfills the linear growth condition $|\psi(v)|\le C(1+\|v\|_{L^2(\bbT;H)})$ for some deterministic constant $C=C(\psi)>0$ and all $v\in L^2(\bbT;H)$. Hence, we have
	\begin{equation*}
	\bbE(\Psi^r)\le \bbE\Big(C^r(1+\|u\|_{L^2(\bbT;V)})^r\Big)\le C^r 2^{r-1}\Big(1+\bbE\Big(\sup_{t\in\bbT}\|u\|^r_{*,t}\Big)\Big)<+\infty.	
	\end{equation*}
\end{proof}

\section{Numerical Approximation of the Solution}\label{astein_sec:para_disc}

In general, the (exact) weak solution $u$ to Problem~\eqref{astein_eq:pde} is out of reach and we have to find tractable approximations of $u$ to apply Monte Carlo algorithms for the estimation of $\bbE(\Psi)$. A common approach is to use a FE discretization of $V$ combined with a time marching scheme to sample path-wise approximations of $u$. For this, however, it is necessary to evaluate $a$ and $b$ at certain points in $\bbD$. This is in general infeasible, since the Gaussian field $W$ usually involves an infinite series and/or the jump heights $P_i$ might not be sampled without bias. The latter issue may arise if $P_i$ has non-standard law, e.g. the \textit{generalized inverse Gaussian} distribution, for more details we refer to \cite{BS18b, BS18c}.
We may circumvent this issue by constructing suitable approximations of $a$ and $b$, for instance by truncated Karhunen-Lo{\`e}ve expansions (\cite{C12, CST13}), circulant embedding methods (\cite{Gr18a, LP11}) or Fourier inversion techniques for the sampling of $P_i$ (\cite{BS18a, BS18b}).
Hence, we obtain a modified problem with approximated coefficients which may then be discretized in the spatial and temporal domain. 
To increase the order of convergence in the spatial discretization, we introduce a FE scheme in the second part of this section where we choose the FE grids adapted with respect to the discontinuities in each sample of $a$ and $b$.    
Under mild assumptions on the coefficients we then derive errors on the semi- and fully discrete approximations of $u$.  

\subsection{Approximated Diffusion Coefficients}
As discussed above, there are several methods available to obtain tractable approximations of the diffusion coefficient $a$, thus we consider a rather general setting here. 
For some $\epsilon>0$, let $a_\epsilon:\Omega\times \bbD\to\R_{>0}$ be an arbitrary approximation of the diffusion coefficient and let (according to Definition~\ref{astein_def:a}) 
\begin{equation*}
b_\epsilon:\Omega\times \bbD\to\R,\quad (\omega,x)\mapsto \min(b_1(x)a_\epsilon(\omega,x),b_2(x)),
\end{equation*}
be the canonical approximation of $b$.
Substituting $a_\epsilon$ and $b_\epsilon$ into Problem~\eqref{astein_eq:pde} yields  
\begin{align}\label{astein_eq:pde_approx}
\begin{split}
\partial_t u_\epsilon(\omega,x,t)+[L_\epsilon u_\epsilon](\omega,x,t)&=f(\omega,x,t)\quad\text{in $\Omega\times \bbD\times (0,T]$},\\
u_\epsilon(\omega,x,0)&=u_0(\omega,x)\quad\text{in $\Omega\times \bbD\times\{0\}$}\\
u_\epsilon(\omega,x,t)&=0\quad\text{on $\Omega\times\partial \bbD\times\bbT$},
\end{split}
\end{align}
where the approximated second order differential operator $L_\epsilon $ is given by 
\begin{equation*}
[L_\epsilon u](\omega,x,t)=-\nabla\cdot \left(a_\epsilon(\omega,x)\nabla u(\omega,x,t)\right)+b_\epsilon(\omega,x)\mathbf 1^T\nabla u(\omega,x,t).
\end{equation*}
The path-wise variational formulation of Eq.~\eqref{astein_eq:pde_approx} is then (analogous to Eq.~\eqref{astein_eq:var}) given by: 
For almost all $\omega\in\Omega$ with given $f(\omega,\cdot,\cdot)$, find $u_\epsilon(\omega,\cdot,\cdot)\in L^2(\bbT;V)$ with $\partial_t u(\omega,\cdot,\cdot)\in L^2(\bbT;V')$ such that, for $t\in\bbT$,
\begin{equation}\label{astein_eq:var_approx}
\dualpair{V'}{V}{\partial_t u_\epsilon(\omega,\cdot,t)}{v}+B_{\epsilon,\omega}(u_\epsilon(\omega,\cdot,t),v)=F_{\omega,t}(v),
\end{equation}
holds for all $v\in V$ with respect to the approximated bilinear form 
\begin{align*}
B_{\epsilon,\omega}(v,w):=\int_ \bbD a_\epsilon(\omega,x)\nabla v(x)\cdot\nabla w(x)+b_\epsilon(\omega,x)\mathbf 1^T\nabla v(x)w(x)dx,\quad v,w\in V.
\end{align*}

The following assumption guarantees existence and uniqueness of $u_\epsilon$ and allows us to bound $u-u_\epsilon$ in a mean-square sense.

\begin{assumption}\label{astein_ass:ahat}
	Let Assumption~\ref{astein_ass:EV} hold and let $a_\epsilon:\Omega\times \bbD\to\R_{>0}$ be an approximation of $a$ for some fixed $\epsilon>0$.
	Define $a_{\epsilon,-}(\omega):=\text{ess inf} a_\epsilon(\omega,x)$ and $a_{\epsilon,+}(\omega):=\text{ess sup}_{x\in \bbD} a_\epsilon(\omega,x)$.
	Assume that for some $s>(1/2-1/p)^{-1}$ and any $q\in[1,\infty)$, there are constants $C_i>0$, for $i=1,\dots,4$, independent of $\epsilon$, such that 
	\begin{itemize}
		\item $\|a-a_\epsilon\|_{L^s(\Omega;L^\infty( \bbD))}\le C_1\epsilon$,
		\item $\|1/a_{\epsilon,-}\|_{L^q(\Omega;\R)}\le C_2\|1/a_-\|_{L^q(\Omega;\R)}<+\infty$,
		\item $\|a_{\epsilon,+}\|_{L^q(\Omega;\R)}\le C_3\|a_+\|_{L^q(\Omega;\R)}<+\infty$\quad and
		\item $\|\max\limits_{i=1,\dots,\tau}\sum_{j=1}^d\|\partial_{x_j}a_\epsilon\|_{L^\infty( \calT_i)}\|_{L^q(\Omega;\R)}\le C_4\|\max\limits_{i=1,\dots,\tau}\sum_{j=1}^d\|\partial_{x_j}a\|_{L^\infty( \calT_i)}\|_{L^q(\Omega;\R)}<+\infty$.
	\end{itemize}
\end{assumption}

At this point we remark that Assumption~\ref{astein_ass:ahat} is natural and essentially states that $a_\epsilon$ has the same regularity as $a$.
Furthermore, the moments of $a-a_\epsilon$ are controlled by the parameter $\epsilon$ and we may achieve an arbitrary good approximation by choosing $\epsilon$ sufficiently small. 
This holds for instance (with $C_2=C_3=C_4=1$) if $W$ is approximated by a truncated Karhunen-Lo\`eve expansion (see \cite{BS18b, BS18c}) or if $a_\epsilon$ stems from linear interpolation of discrete sample points of $W$ as we explain in Section~\ref{astein_sec:num}.
\begin{theorem}\label{astein_thm:a_error_V}
	Let Assumption~\ref{astein_ass:ahat} hold and let $u_\epsilon$ be the weak solution to Problem~\eqref{astein_eq:pde_approx}.
	Then, the root-mean-squared approximation error is bounded by
	\begin{equation*}
	\bbE\Big(\sup_{t\in\bbT}\|u(\cdot,\cdot,t)-u_\epsilon(\cdot,\cdot,t)\|^2_{*,t}\Big)^{1/2}\le C\epsilon.
	\end{equation*}
\end{theorem}
\begin{proof}
	By Theorem~\ref{astein_thm:exis}, we have existence of unique solutions $u$ and $u_\epsilon$ to Eqs.~\eqref{astein_eq:var} resp.~\eqref{astein_eq:var_approx} almost surely. Thus, we obtain the variational problem: Find $u-u_\epsilon$ such that 
	\begin{align*}
	\dualpair{V'}{V}{\partial_t (u(\omega,\cdot,t)-u_\epsilon(\omega,\cdot,t))}{v}+B_\omega(u(\omega,\cdot,t)-u_\epsilon(\omega,\cdot,t),v)
	=\dualpair{V'}{V}{\widetilde f(\omega,\cdot,t)}{v}
	\end{align*}
	for all $t\in\bbT$ and  $v\in V$ with initial condition $(u-u_\epsilon)(\cdot,\cdot,0)\equiv0$ and right hand side 
	\begin{align*}
	\widetilde f(\omega,\cdot,t):=\nabla\cdot((a_\epsilon-a)(\omega,\cdot)\nabla u_\epsilon(\omega,\cdot,t))+(b_\epsilon-b)(\omega,\cdot)\mathbf 1^T\nabla u_\epsilon(\omega,\cdot,t)\in V'.
	\end{align*}
	By H\"older's inequality it holds
	\begin{align*}
	\|\widetilde f(\omega,\cdot,\cdot)\|_{L^2(\bbT;V')}
	&\le \|(a-a_\epsilon)(\omega,\cdot)\|_{L^\infty( \bbD)}\|\|\nabla u(\omega,\cdot,\cdot)\|_2\|_{L^2(\bbT;H)}\\
	&\quad+\|(b-b_\epsilon)(\omega,\cdot)\|_{L^\infty( \bbD)}\|\mathbf 1^T\nabla u(\omega,\cdot,\cdot)\|_{L^2(\bbT;H)}\\
	&\le C(1+\|b_1\|_{L^\infty(\bbD)})\|(a-a_\epsilon)(\omega,\cdot)\|_{L^\infty( \bbD)}\|\|\nabla u(\omega,\cdot,\cdot)\|_2\|_{L^2(\bbT;H)},
	\end{align*}
	which yields using Assumption~\ref{astein_ass:ahat} and Theorem~\ref{astein_thm:exis}
	\begin{align*}
	\|\widetilde f(\omega,\cdot,\cdot)\|_{L^{p_1}(\Omega;L^2(\bbT;V'))}
	&\le C(1+\|b_1\|_{L^\infty(\bbD)})\|(a-a_\epsilon)\|_{L^s(\Omega;L^\infty( \bbD))}\bbE\Big(\sup_{t\in\bbT}\|u\|_{*,t}^r\Big)^{1/r}\\
	&\le C\epsilon
	\end{align*}
	for $r\in((1/2-1/s)^{-1},p)$ and $p_1:=(1/s+1/r)^{-1}>2$. 
	We may now use Theorem~\ref{astein_thm:exis} with $q=(1/2-1/p_1)^{-1}$ to estimate $u-u_\epsilon$ via
	\begin{equation*}
	\bbE\Big(\sup_{t\in\bbT} \|u-u_\epsilon\|^2_{*,t}\Big)^{1/2}\le C\|1/a_-\|_{L^q(\Omega;\R)}\|\widetilde f\|_{L^{p_1}(\Omega;L^2(\bbT;V'))}\le C\epsilon.
	\end{equation*}
\end{proof}

\subsection{Semi-Discretization by Adaptive Finite Elements}
Given a suitable approximation $a_\epsilon $ of the diffusion coefficient, we discretize the (approximate) solution $u_\epsilon$ in the spatial domain. As a first step, we replace the (infinite-dimensional) solution space $V$ by a sequence $\bbV=(V_\ell, \ell\in\N_0)$ of finite dimensional subspaces $V_\ell\subset V$.
In general, $V_\ell$ are standard FE spaces of piecewise linear functions with respect to some given triangulation $\rmK_\ell$ of $ \bbD$ and $h_\ell$ represents the maximum diameter of $\rmK_\ell$. 
As indicated in \cite{BS18b, BS18c} using standard FE spaces will not yield the full order of convergence with respect to $h_\ell$ due to the discontinuities in $a_\epsilon$ and $b_\epsilon$. Thus, we follow the same approach as in \cite{BS18b} for Problem~\eqref{astein_eq:var_approx}
and utilize path-dependent meshes to match the interfaces generated by the jump-diffusion and -advection coefficients. 
As this entails changing varying approximation spaces $V_\ell$  with each sample of $a_\epsilon$ resp. $b_\epsilon$, we have to formulate a semi-discrete version of problem~\eqref{astein_eq:var_approx} with respect to $\omega\in \Omega$:

Given a fixed $\omega\in\Omega$ and $\ell\in\N_0$, we consider a (stochastic) finite dimensional subspace $V_\ell(\omega)\subset V$ with sample-dependent basis $\{v_1(\omega),\dots,v_{d_\ell}(\omega)\}\subset V$ and stochastic dimension $d_\ell=d_\ell(\omega)\in\N$. 
For a given random partition $\calT(\omega)=(\calT_i,i=1\dots,\tau(\omega))$ of polygons on $ \bbD$, we choose a conforming triangulation $\rmK_{\ell}(\omega)$ such that 
\begin{equation*}
\calT(\omega)\subset\rmK_{\ell}(\omega)\;\text{ and }\; h_{\ell}(\omega) :=\max_{K\in\rmK_{\ell}(\omega)}\text{diam}(K)\le \overline h_\ell\;\text{ for $\ell\in\N_0$,}
\end{equation*}
holds almost surely. 
The inclusion $\calT(\omega)\subset\rmK_{\ell}(\omega)$ states that the triangles in $\rmK_\ell(\omega)$ are chosen to match and fully cover the polygonal partition elements in $\calT(\omega)$.
Furthermore, $(\overline h_\ell,\ell\in\N_0)$ is a sequence of positive, deterministic refinement thresholds, decreasing monotonically to zero.
This guarantees that $h_\ell(\omega)\to0$ for $\ell\to\infty$ almost surely, although the absolute speed of convergence varies for each $\omega$. 
We assume shape-regularity of the triangulation uniform in $\Omega$, i.e. there exist a $\vartheta\in (0,1)$ such that
\begin{equation*}\label{astein_eq:angle}
0<\vartheta\le\sup_{\ell\in\N_0}\sup_{K\in K_\ell(\omega)}\frac{\text{diam}(K)}{\iota_K}\le\vartheta^{-1}<+\infty\quad\text{almost surely.}
\end{equation*}
In Ineq.~\eqref{astein_eq:angle}, $\iota_T$ denotes the diameter of the inscribed circle of the triangle $K$. 
For given $\{v_1(\omega),\dots,v_{d_\ell}(\omega)\}$, the semi-discrete version of the variational formulation~\eqref{astein_eq:var_approx} is then to find $u_{\epsilon,\ell}(\omega,\cdot,t)\in V_\ell(\omega)$ such that for $t\in\bbT$ and $v_\ell(\omega)\in V_\ell(\omega)$
\begin{equation}\label{astein_eq:semi_var}
\begin{split}
	\dualpair{V'}{V}{\partial_t u_{\epsilon,\ell}(\omega,\cdot,t)}{v_\ell(\omega)}
	+B_{\epsilon,\omega}(u_{\epsilon,\ell}(\omega,\cdot,t),v_\ell(\omega))
	&=\dualpair{V'}{V}{f(\omega,\cdot,t)}{v_\ell(\omega)},\\
	u_{\epsilon,\ell}(\omega,\cdot,0)&=u_{0,\ell}( \omega,\cdot),
\end{split}
\end{equation}
where $u_{0,\ell}(\omega,\cdot)\in V_\ell(\omega)$ is a suitable approximation of $u_0(\omega,\cdot)$, for instance the nodal interpolation of $u_0$ in $V_\ell(\omega)$.
The function $u_{\epsilon,\ell}(\omega,\cdot,t)$ may be expanded as
\begin{equation*}
u_{\epsilon,\ell}(\omega,\cdot,t)=\sum_{j=1}^{d_\ell(\omega)}c_j(\omega,t)v_j(\omega),
\end{equation*}
where the coefficients $c_1(\omega,t),\dots,c_{d_\ell}(\omega,t)\in\R$ depend on $(\omega,t)\in\Omega\times\bbT$ and the respective coefficient (column-)vector is
${\bf c(\omega,t)} := (c_1(\omega,t),\dots,c_{d_\ell}(\omega,t))^T$. With this, the semi-discrete variational problem in the (stochastic) finite dimensional space $V_\ell(\omega)$ is equivalent to solving the system of ordinary differential equations
\begin{equation}\label{astein_eq:ode_sd}
\frac{d}{dt}\bf{c}(\omega,t)+ \bf{A} (\omega) {\bf{c} (\omega,t)}
=\bf{F}(\omega,t),\quad t\in\bbT
\end{equation}
for $\bf c$ with stochastic stiffness matrix $(\mathbf A(\omega))_{jk}=B_{\epsilon,\omega}(v_j(\omega),v_k(\omega))$ and time-dependent load vector $(\mathbf F(\omega,t))_j=\dualpair{V'}{V}{f(\omega,\cdot,t)}{v_j(\omega)}$ for $j,k\in\{1,\dots,d_\ell(\omega)\}$.
The following result gives an error estimate in the energy norm for $u_\epsilon-u_{\epsilon,\ell}$.
\begin{theorem}\cite[Theorem 4.7]{BS18c}\label{astein_thm:semi_error_V}
	Let Assumption~\ref{astein_ass:ahat} hold such that for some $\kappa\in(1/2,1]$ it holds that
	$\bbE(\max_{i=1,\dots,\tau} \|u\|^2_{H^{1+\kappa}(\calT_i)})<+\infty$.
	 Let $u_{\epsilon,\ell}$ be the semi-discrete sample-adapted approximation of $u_{\epsilon}$ as in Eq.~\eqref{astein_eq:semi_var} and let $\|(u_0-u_{\ell,0})(\omega,\cdot)\|_H\le C\|u_0(\omega,\cdot)\|_V\overline h_\ell$ almost surely for all $\ell\in\N_0$. 
	Then, there holds almost surely the path-wise estimate 
	\begin{equation*}
	\sup_{t\in\bbT}\|(u_\epsilon-u_{\epsilon,\ell})(\omega,\cdot,\cdot)\|_{*,t}
	\le C/(a_{\epsilon,-}(\omega))^{1/2}\Big(\|f(\omega,\cdot,\cdot)\|_{L^2(\bbT;H)}+\|u_0(\omega,\cdot)\|_V\Big)\overline h_\ell^\kappa
	\end{equation*}
	and, for any $r\in[1,p)$ (with $p$ as in Ass~\ref{astein_ass:EV}), the expected parabolic estimate
	\begin{align*}
	\bbE(\sup_{t\in\bbT}\|u_\epsilon-u_{\epsilon,\ell}\|_{*,t}^r)^{1/r}\le C(\|f\|_{L^p(\Omega;L^2(\bbT;H))}+\|u_0\|_{L^p(\Omega;V)})\overline h_\ell^\kappa.
	\end{align*}
\end{theorem}
The above statement gives a bound on the error in the $L^2(\bbT;V)$-norm. The functional $\Psi$ however is defined on $L^2(\bbT;H)$, thus it is favorable to derive an error bound with respect to the weaker $L^2(\bbT;H)$-norm.
\begin{theorem}\label{astein_thm:semi_error_H}
	Let Assumption~\ref{astein_ass:ahat} hold such that for some $\kappa\in(1/2,1]$ there holds
	$\bbE(\max_{i=1,\dots,\tau} \|u\|^2_{H^{1+\kappa}(\calT_i)})<+\infty$ and let $\|(u_0-u_{\ell,0})(\omega,\cdot)\|_H\le C\|u_0(\omega,\cdot)\|_{H^2(\bbD)}\overline h_\ell^2$ almost surely. Then, 
	\begin{align*}
	\bbE(\|u_\epsilon-u_{\ell,\epsilon}\|_{L^2(\bbT;H)}^2)^{1/2}\le C\overline h_\ell^{2\kappa}.
	\end{align*}
\end{theorem}
\begin{proof}
	For fixed $\omega$,	we consider the path-wise parabolic dual problem to find $w(\omega,\cdot,\cdot)\in L^2(\bbT; V)$ with $\partial_t w(\omega,\cdot,\cdot)\in L^2(\bbT; V')$ such that, for $t\in\bbT$,
	\begin{equation}\label{astein_eq:var_dual1}
	\dualpair{V'}{V}{\partial_t w(\omega,\cdot,t)}{v}+B_{\epsilon,\omega}(w(\omega,\cdot,t),v)=\dualpair{V'}{V}{g(\omega,\cdot,t)}{v},\quad \text{ for all } v\in V,
	\end{equation}
	where $w(\omega,\cdot,0)=w_0(\omega,\cdot):=0$ and $g(\omega,\cdot,t):=(u_{\epsilon}-u_{\epsilon,\ell})(\omega,\cdot,T-t)\in V$ almost surely for any $t\in\bbT$ by Theorem~\ref{astein_thm:exis}.
	Hence, we may test against $v=g(\omega,\cdot,t)$ in Eq.~\eqref{astein_eq:var_dual1} to obtain
	\begin{equation}\label{astein_eq:var_dual2}
	\|g(\omega,\cdot,t)\|_H^2=\dualpair{V'}{V}{\partial_t w(\omega,\cdot,t)}{g(\omega,\cdot,t)}+B_{\epsilon,\omega}(w(\omega,\cdot,t),g(\omega,\cdot,t)).
	\end{equation}
	Furthermore, for any $v_\ell(\omega)\in V_\ell(\omega)$ it holds by Eqs.~\eqref{astein_eq:var_approx},\eqref{astein_eq:semi_var}
	\begin{equation}\label{astein_eq:var_dual3}
	\dualpair{V'}{V}{\partial_t (u_\epsilon-u_{\epsilon,\ell})(\omega,\cdot,t)}{v_\ell(\omega)}
	=-B_{\epsilon,\omega}((u_\epsilon-u_{\epsilon,\ell})(\omega,\cdot,t),v_\ell(\omega))
	\end{equation} 
	and thus
	\begin{equation}\label{astein_eq:time_change}
	\begin{split}
	B_{\epsilon,\omega}(g(\omega,\cdot,t),w(\omega,\cdot,t))
	&=\dualpair{V'}{V}{\partial_t g(\omega,\cdot,t)}{v_\ell(\omega)-w(\omega,\cdot,t)+w(\omega,\cdot,t)}\\
	&\quad+B_{\epsilon,\omega}(g(\omega,\cdot,t),w(\omega,\cdot,t)-v_\ell(\omega)),
	\end{split} 
	\end{equation}	where we have used the that
	$\partial_tg(\omega,\cdot,t)=-(\partial_tu_{\epsilon}-\partial_tu_{\epsilon,\ell})(\omega,\cdot,T-t)$ by the chain rule.
	Substituting Eq.~\eqref{astein_eq:time_change} in Eq.~\eqref{astein_eq:var_dual2} and integrating over $\bbT$ yields
	\begin{align*}
	\|g(\omega,\cdot,\cdot)\|_{L^2(\bbT;H)}^2
	&=\int_0^T\dualpair{V'}{V}{\partial_t w(\omega,\cdot,t)}{g(\omega,\cdot,t)}+\dualpair{V'}{V}{\partial_tg(\omega,\cdot,t)}{w(\omega,\cdot,t)}dt\\
	&\quad+\int_0^T\dualpair{V'}{V}{\partial_tg(\omega,\cdot,t)}{v_\ell(\omega)-w(\omega,\cdot,t)}dt\\
	&\quad+\int_0^TB_{\epsilon,\omega}(g(\omega,\cdot,t),w(\omega,\cdot,t)-v_\ell(\omega))dt\\
	&=:I+II+III.
	\end{align*}
	Integration by parts and the path-wise estimate in Theorem~\ref{astein_thm:exis} yield for $I$
	\begin{align*}
	I&=(w(\omega,\cdot,T),g(\omega,\cdot,T))_H-(w_0(\omega,\cdot),g(\omega,\cdot,0))_H\\
	&\le \|w(\omega,\cdot,T)\|_H\|u_0(\omega,\cdot)-u_{0,\ell}(\omega,\cdot)\|_H\\
	&\le C\frac{1}{a_{\epsilon,-}(\omega)}\|g(\omega,\cdot,\cdot)\|_{L^2(\bbT;H)}\|u_0(\omega,\cdot)\|_{H^2(\bbD)}\overline h_\ell^2,
	\end{align*}
	where we have used $\|(u_0-u_{\ell,0})(\omega,\cdot)\|_H\le C\|u_0(\omega,\cdot)\|_{H^2(\bbD)}\overline h_\ell^2$ in the last step.
	To bound the second term, we choose $v_\ell=v_\ell(\omega,\cdot,t)$ to be the semi-discrete FE approximation of $w(\omega,\cdot,t)$ in $V_\ell(\omega)$.
	Since $w_0\equiv 0$, there is no approximation error in the initial condition and with the path-wise estimate from Theorem~\ref{astein_thm:semi_error_V} it follows that
	\begin{align*}
	II&\le \|\partial_tg(\omega,\cdot,\cdot)\|_{L^2(\bbT;V')}\|v_\ell(\omega,\cdot,\cdot)-w(\omega,\cdot,\cdot)\|_{L^2(\bbT;V)}\\
	&\le C\frac{1}{(a_{\epsilon,-}(\omega))^{1/2}}\|\partial_tg(\omega,\cdot,\cdot)\|_{L^2(\bbT;V')}\|g(\omega,\cdot,\cdot)\|_{L^2(\bbT;H)}\overline h_\ell^\kappa.
	\end{align*}
	From Eq.~\eqref{astein_eq:var_dual3} and Theorem~\ref{astein_thm:semi_error_V} we also see that 
	\begin{align*}
	\|\partial_tg(\omega,\cdot,\cdot)\|_{L^2(\bbT;V')}
	\le C\frac{a_{\epsilon,+}(\omega)}{(a_{\epsilon,-}(\omega))^{1/2}}\Big(\|f(\omega,\cdot,\cdot)\|_{L^2(\bbT;H)}+\|u_0(\omega,\cdot)\|_{V}\Big)\overline h_\ell^\kappa
	\end{align*}
	and thus
	\begin{align*}
	II\le C\frac{a_{\epsilon,+}(\omega)}{a_{\epsilon,-}(\omega)}\Big(\|f(\omega,\cdot,\cdot)\|_{L^2(\bbT;H)}+\|u_0(\omega,\cdot)\|_{V}\Big)\|g(\omega,\cdot,\cdot)\|_{L^2(\bbT;H)}\overline h_\ell^{2\kappa}.
	\end{align*}
	Similarly, we bound the last term again with Theorem~\ref{astein_thm:semi_error_V} via 
	\begin{align*}
	III&\le Ca_{\epsilon,+}(\omega)	\|g(\omega,\cdot,\cdot)\|_{L^2(\bbT;V)}\|v_\ell(\omega,\cdot,\cdot)-w(\omega,\cdot,\cdot)\|_{L^2(\bbT;V)}\\
	&\le C\frac{a_{\epsilon,+}(\omega)}{a_{\epsilon,-}(\omega)}
	\Big(\|f(\omega,\cdot,\cdot)\|_{L^2(\bbT;H)}+\|(u_0(\omega,\cdot)\|_{V}\Big)\|g(\omega,\cdot,\cdot)\|_{L^2(\bbT;H)}\overline h_\ell^{2\kappa}.
	\end{align*}	
	The estimates on $I-III$ now show that 
	\begin{align*}
	\|g(\omega,\cdot,\cdot)\|_{L^2(\bbT;H)}
	&\le C\frac{a_{\epsilon,+}(\omega)}{a_{\epsilon,-}(\omega)}
	\Big(\|f(\omega,\cdot,\cdot)\|_{L^2(\bbT;H)}+\|(u_0(\omega,\cdot)\|_{H^2(\bbD)}\Big)\overline h_\ell^{2\kappa}.
	\end{align*}
	and the claim follows by Assumption~\ref{astein_ass:ahat} and H\"older's inequality.
\end{proof}

\begin{remark}
We remark that the additional condition on the initial data approximation in Theorem~\ref{astein_thm:semi_error_H} is fulfilled if $u_0$ has almost surely continuous paths and $u_{\ell,0}$ is chosen as the path-wise nodal interpolation with respect to the sample-adapted FE basis.
\end{remark}

\subsection{Fully Discrete Pathwise Approximation}
For a fully discrete formulation of Problem~\eqref{astein_eq:semi_var}, we consider a time grid $0=t_0<t_1<\dots<t_n=T$ in $\bbT$ for some $n\in\N$
and assume the grid is equidistant with fixed time step $\Delta t:=t_i-t_{i-1}>0$. 
The temporal derivative at $t_i$ is approximated by the backward difference 
\begin{equation*}
\partial_t u_{\epsilon,\ell}(\omega,\cdot,t_i)=(u_{\epsilon,\ell}(\omega,\cdot,t_i)- u_{\epsilon,\ell}(\omega,\cdot,t_{i-1}))/\Delta t,\quad i=1,\dots,n.
\end{equation*}
We emphasize again that in our model problem the weak and strong temporal derivative of $u_{\epsilon,\ell}$ coincide due to the temporal regularity of the solution. Hence, the backward difference as an approximation scheme in a strong sense is justified. 
This yields the fully discrete problem to find $(u_{\epsilon,\ell}^{(i)}(\omega,\cdot),i=0,\dots,n)\subset V_\ell(\omega)$ such that for all $ v_\ell(\omega)\in V_\ell(\omega)$ and $i=1,\dots,n$
\begin{equation*}
\begin{split}
	\frac{((u_{\epsilon,\ell}^{(i)}-u_{\epsilon,\ell}^{(i-1)})(\omega,\cdot), v_\ell(\omega))_H}{\Delta t}+B_{\epsilon,\omega}( u_{\epsilon,\ell}^{(i)}(\omega,\cdot),v_\ell(\omega))
	&=\dualpair{V'}{V}{f(\omega,\cdot,t_i)}{v_\ell(\omega)}, \\
	u_{\epsilon,\ell}^{(0)}(\omega,\cdot)&=u_{0,\ell}(\omega,\cdot).
\end{split}
\end{equation*}
The fully discrete solution is given by  
\begin{equation*}
u_{\epsilon,\ell}^{(i)}(\omega,\cdot)=\sum_{j=1}^{d_\ell(\omega)}c_{i,j}(\omega)v_j(\omega),\quad i=1,\dots,n,
\end{equation*}
where the coefficient vector $\mathbf{c_i}(\omega)=(c_{i,1}(\omega),\dots,c_{i,d_\ell}(\omega)))$ solves the linear system of equations
\begin{equation*}
(\mathbf M+\Delta t\mathbf A(\omega)) {\bf c_i(\omega)}
=\Delta t\mathbf F(\omega,t_i)+\mathbf{Mc_{i-1}}(\omega)
\end{equation*}
in every discrete point in time $t_i$, and $\mathbf A$ and $\mathbf F$ are as in Eq.~\eqref{astein_eq:ode_sd}.
The mass matrix is given by $(\mathbf M)_{jk}:=(v_j(\omega),v_k(\omega))_H$ and $ \mathbf{c_0}$ consists of the basis coefficients of $u_{0,\ell}\in V_\ell(\omega)$ with respect to $\{v_1(\omega),\dots,v_{d_\ell}(\omega)\}$.
We extend the discrete solution to the whole temporal domain by the linear interpolation 
\begin{equation*}
\overline u_{\epsilon,\ell}(\cdot,\cdot,t):=(u_{\epsilon,\ell}^{(i)}-u_{\epsilon,\ell}^{(i-1)})\frac{(t-t_{i-1})}{\Delta t}+u_{\epsilon,\ell}^{(i-1)},\quad t\in[t_{i-1},t_i],\quad i=1,\dots,n.
\end{equation*}

\begin{theorem}\cite[Theorem 4.12]{BS18c}\label{astein_thm:full_error_V}
	Let Assumption~\ref{astein_ass:ahat} hold, let $(u_{\epsilon,\ell}^{(i)},i=0,\dots,n)$ be the fully discrete sample-adapted approximation of $u_{N,\epsilon}$, and let $\overline u_{\epsilon,\ell}$ be the linear interpolation of $(u_{\epsilon,\ell}^{(i)},i=0,\dots,n)$ in $\bbT$. 
	Then, for $C>0$ independent of $\epsilon, h_\ell$ and $\Delta t$, it holds
	\begin{align*}
	\bbE(\sup_{t\in\bbT}\|u_{\epsilon,\ell}-\overline u_{\epsilon,\ell}\|_{*,t}^2)^{1/2}\le C\Delta t.
	\end{align*}
\end{theorem}
The final corollary on the overall approximation error is now an immediate consequence of Theorems~\ref{astein_thm:a_error_V},\,\ref{astein_thm:semi_error_H} and \ref{astein_thm:full_error_V} and the Lipschitz condition on $\psi$.
\begin{corollary}\label{astein_cor:overall_error}
	Let Assumption~\ref{astein_ass:ahat} hold such that for some $\kappa\in(1/2,1]$ there holds
	$\bbE(\max_{i=1,\dots,\tau} \|u\|^2_{H^{1+\kappa}(\calT_i)})<+\infty$ and let  $\|(u_0-u_{\ell,0})(\omega,\cdot)\|_H\le C\|u_0(\omega,\cdot)\|_{H^2(\bbD)}\overline h_\ell^2$ almost surely.
	The (fully) approximated QoI is defined by $\Psi_{\epsilon,\ell,\Delta t}:=\psi(\overline u_{\epsilon,\ell})$. 
	Then, there holds the error bound
	\begin{equation*}
	\EE(|\Psi-\Psi_{\epsilon,\ell,\Delta t}|^2)^{1/2}\le C(\epsilon+\overline h_\ell^{2\kappa}+\Delta t).
	\end{equation*}
\end{corollary}

Given a sequence of discretization tresholds $\overline h_\ell>0$ for $\ell\in\N_0$, one should adjust $\epsilon$ and $\Delta t$ such that $\overline h_\ell^{2\kappa}\simeq\epsilon\simeq\Delta t$ to achieve an error equilibrium. Hence, we denote the adjusted parameters on level $\ell$ by $\epsilon_\ell$ and $\Delta t_\ell$ and assume that all errors are equilibrated in the sense that $c\overline h_\ell^{2\kappa}\le \epsilon_\ell,\Delta t_\ell\le C \overline h_\ell^{2\kappa}$ holds
for constants $c,C>0$ independent of $\ell$.
We further define $\Psi_\ell:=\Psi_{\epsilon_\ell,\ell,\Delta t_\ell}=\psi(\overline u_{\epsilon_\ell,\ell})$ and obtain with Corollary~\ref{astein_cor:overall_error}
\begin{equation} \label{astein_eq:psi_rmse}
\bbE(|\Psi-\Psi_\ell|^2)^{1/2}\le C\overline h_\ell^{2\kappa}.
\end{equation}

\section{Estimation of Moments by Multilevel Monte Carlo Methods}\label{astein_sec:mlmc}
As we are able to generate samples from $\Psi_\ell=\psi(\overline u_{\epsilon_\ell,\ell})$ and control for the discretization error in each sample,  
we may estimate the expectation $\bbE(\Psi)$ by Monte Carlo methods. For convenience, we restrict ourselves to the estimation of $\bbE(\Psi)$, but we note that all results from this section are valid when estimating higher moments of $\Psi$, given that $u\in L^r(\Omega;L^2(\bbT;V))$ for sufficiently high $r$ (cf. Theorem~\ref{astein_thm:exis}).
Our focus is on multilevel Monte Carlo (MLMC) estimators, since they are easily implemented, do not require much regularity of $\Psi$ and are significantly more efficient than standard Monte Carlo  estimators. The main idea of the MLMC estimation has been developed in \cite{H01} and later been rediscovered and popularized in \cite{G08}.
In this section, we briefly recall the MLMC method and then show how we achieve a desired error rate by adjusting the number of samples on each level to the discretization bias. 
We also suggest a modification of the MLMC algorithm to increase computational efficiency before we verify our results in Section~\ref{astein_sec:num}.

Let $L\in\N$ be a fixed (maximum) discretization level and assume that the approximation parameters on each level $\ell=0,\dots,L$ satisfy $\overline h_\ell^{2\kappa}\simeq\epsilon_\ell\simeq\Delta t_\ell$ (see Section~\ref{astein_sec:para_disc}). This yields a sequence $\Psi_0,\dots,\Psi_L$ of approximated QoIs, hence the \textit{MLMC  estimator} of $\EE(\Psi_L)$ is given by 
\begin{equation}\label{astein_eq:mlmc}
E^L(\Psi_L)=\sum_{\ell=0}^L\frac{1}{M_\ell}\sum_{i=1}^{M_\ell}\Psi^{(i,\ell)}_\ell-\Psi^{(i,\ell)}_{\ell-1},
\end{equation}
where we have set $\Psi_{-1}:=0$.
Above, $(\Psi^{(i,\ell)}_\ell-\Psi^{(i,\ell)}_{\ell-1},i\in\N)$ is a sequence of independent copies of $\Psi_\ell-\Psi_{\ell-1}$ and $M_\ell\in\N$ denotes the number of samples on each level.
To achieve a desired target root mean-squared error (RMSE), this estimator requires less computational effort than the standard Monte Carlo approach under certain assumptions.
This, by now, classical result was proven in \cite[Theorem 3.1]{G08} for functionals of stochastic differential equations. 
The proof is rather general and may readily be transferred to other applications, for instance the estimation of functionals or moments of random PDEs, see \cite{BSZ11, G15}.

\begin{theorem}\label{astein_thm:mlmc}
	Let Assumption~\ref{astein_ass:ahat} hold such that for some $\kappa\in(1/2,1]$ there holds
	$\bbE(\max_{i=1,\dots,\tau} \|u\|^2_{H^{1+\kappa}(\calT_i)})<+\infty$  and let $\overline h_{\ell-1}\le C_1 \overline h_\ell$ for some $C_1>0$ for all $\ell\in\N_0$. 
	For $L\in\N$ and given refinement parameters $\overline h_0>\dots>\overline h_L>0$ choose $\Delta t_\ell, \epsilon_\ell>0$ such that  
	$\epsilon_\ell,\Delta t_\ell \le  C_2\overline h_\ell^{2\kappa}$ holds for fixed $C_2>0$ and $\ell=0,\dots,L$.
	Furthermore, let $(\rho_\ell,\ell=1,\dots,L)\in (0,1)^L$ be a set of positive weights such that $\sum_{\ell=1}^L\rho_\ell=C_\rho$, with a constant $C_\rho>0$ independent of $L$, and set
	\begin{equation*}
	M_0^{-1}:=\ceil{\overline h_L^{4\kappa}}\quad\text{and}\quad M_\ell^{-1}:=\ceil{\frac{\overline h_L^{4\kappa}}{\overline h_\ell^{4\kappa}}\rho_\ell^{-2}}\quad\text{for $\ell=1,\dots,L$}.
	\end{equation*}
	Then, there is a $C>0$, independent of $L$ and $\kappa$, such that
	\begin{align*}
	\|\bbE(\Psi)-E^L(\Psi_L)\|_{L^2(\Omega;\R)}\le C\overline h_L^{2\kappa}.
	\end{align*}
\end{theorem}

\begin{proof}
	As all error contributions $\epsilon_\ell,\Delta t_\ell$ are adjusted to $\overline h_\ell$, we obtain by the triangle inequality and Eq.~\eqref{astein_eq:psi_rmse}
	\begin{align*} 
	\|\bbE(\Psi)-E^L(\Psi_L)\|_{L^2(\Omega;\R)}&\le\|\bbE(\Psi)-\bbE(\Psi_L)\|_{L^2(\Omega;\R)}+\|\bbE(\Psi_L)-E^L(\Psi_L)\|_{L^2(\Omega;\R)}\\
	&\le \|\Psi-\Psi_L\|_{L^2(\Omega;\R)}\\
	&\quad+\|\sum_{\ell=0}^L\bbE(\Psi_\ell-\Psi_{\ell-1})-\frac{1}{M_\ell}\sum_{i=1}^{M_\ell}(\Psi^{(i,\ell)}_\ell-\Psi^{(i,\ell)}_{\ell-1})\|_{L^2(\Omega;\R)}\\
	&\le C\overline h_L^{2\kappa} + \sum_{\ell=0}^L\frac{1}{\sqrt{M_\ell}}\|\Psi_\ell-\Psi_{\ell-1}\|_{L^2(\Omega;\R)}.
	\end{align*}
	At this point we emphasize that we did not use the independence of $\Psi^{(i,\ell)}_\ell-\Psi^{(i,\ell)}_{\ell-1}$ across the levels $\ell=1,\dots,L$ in the last inequality.
	We note that 
	\begin{equation*}
	\|\Psi_\ell-\Psi_{\ell-1}\|_{L^2(\Omega;\R)}\le \|\Psi-\Psi_{\ell}\|_{L^2(\Omega;\R)}+ \|\Psi-\Psi_{\ell-1}\|_{L^2(\Omega;\R)}\le C(1+C_1)\overline h_\ell^{2\kappa}
	\end{equation*}
	for $\ell\ge1$ and hence 
	\begin{align*} 
	\|\bbE(\Psi)-E^L(\Psi_L)\|_{L^2(\Omega;\R)}\le C\overline h_L^{2\kappa}+\|\Psi_0\|_{L^2(\Omega;\R)}\overline h_L^{2\kappa}+C(1+C_1)\overline  h_L^{2\kappa}\sum_{\ell=1}^L\rho_\ell
	\le C\overline h_L^{2\kappa}.
	\end{align*}\
\end{proof}

We remark that $C_\rho>0$ may act as a normalizing constant if MLMC estimators based on different discretization techniques are compared, an example is provided in Section~\ref{astein_sec:num}. 
To conclude this section, we briefly present a modified MLMC method to accelerate the estimation of $\bbE(\Psi_L)$. 
In the definition of the MLMC estimator from Eq.~\eqref{astein_eq:mlmc}, the terms in the second sum are independent copies of the corrections $\Psi_\ell-\Psi_{\ell-1}$. 
Hence, one has to generate a total of $M_\ell+M_{\ell+1}$ samples of $\Psi_\ell$ for each $\ell=0,\dots,L$ (where we have set $M_{L+1}:=0$). 
This effort may be reduced if we ``recycle'' the already available samples and generate the differences  
$\Psi^{(i,\ell)}_\ell-\Psi^{(i,\ell)}_{\ell-1}$ and
$\Psi^{(i,\ell)}_{\ell+1}-\Psi^{(i,\ell)}_\ell$
based on the same realization $\Psi^{(i,\ell)}_\ell$.
That is, we drop the second superscript $\ell$ above and arrive at the \textit{coupled MLMC  estimator} 
\begin{equation}\label{astein_eq:mlmc_bs}
E^L_{C}(\Psi_L):=\sum_{\ell=0}^L\frac{1}{M_\ell}\sum_{i=1}^{M_\ell}\Psi^{(i)}_\ell-\Psi^{(i)}_{\ell-1}.
\end{equation}
Instead of $M_\ell+M_{\ell+1}$ realizations of $\Psi_\ell$, the coupled MLMC estimator requires only $M_\ell$ samples of $\Psi_\ell$. 
The copies $\Psi^{(i)}_\ell$ are still independent in $i$, but not anymore across all levels $\ell$ for a fixed index $i$.
Clearly, $\bbE(E^L_{C}(\Psi_L))=\bbE(\Psi_L)$, and it holds
\begin{equation*}
\lim_{L\to+\infty}\bbE(E^L_{C}(\Psi_L))=\lim_{L\to+\infty}\bbE(E^L(\Psi_L))=\lim_{L\to+\infty}\bbE(\Psi_L) = \bbE(u).
\end{equation*}
The introduced modification is a simplified version of the \textit{Multifidelity Monte Carlo estimator} (see~\cite{GPW16}), where the weighting coefficients for all level corrections $\Psi_\ell-\Psi_{\ell-1}$ are set equal to one. 
An estimator similar to~\eqref{astein_eq:mlmc_bs} with coupled correction terms has also been introduced in the context of SDEs in \cite{RG15}. 
As we mentioned in the proof of Theorem~\ref{astein_thm:mlmc}, independence of the sampled differences $\Psi_\ell-\Psi_{\ell-1}$ across $\ell$ is not required for the error estimate, thus, the asymptotic order of convergence also holds for the coupled estimator. 
To compare RMSEs of the estimators from Eq.~\eqref{astein_eq:mlmc} and ~\eqref{astein_eq:mlmc_bs}, we calculate 
\begin{align*}
\var(E_{C}^L(\Psi_L))
&=\var\Big(\sum_{\ell=0}^L
\sum_{i=M_{\ell+1}+1}^{M_\ell}\sum_{k=0}^\ell\frac{\Psi^{(i)}_k-\Psi^{(i)}_{k-1}}{M_k}\Big)\\
&=\sum_{\ell=0}^L(M_{\ell}-M_{\ell+1})\var\left(\sum_{k=0}^{\ell}\frac{\Psi_k-\Psi_{k-1}}{M_k}\right)\\
&=\sum_{\ell=0}^L(M_{\ell}-M_{\ell+1})\Big(\sum_{k=0}^\ell\frac{\bbV_k}{M_k^2}+2\sum_{k=0}^\ell\sum_{j=0}^{k-1}\frac{\bbC_{j,k}}{M_jM_k}\Big)\\
&=\sum_{k=0}^L\Big(\frac{\bbV_k}{M_k^2}+2\sum_{j=0}^{k-1}\frac{\bbC_{j,k}}{M_jM_k}\Big)\sum_{\ell=k}^L(M_{\ell}-M_{\ell+1})\\
&=\var(E^L(\Psi_L))+2\sum_{k=0}^L\sum_{j=0}^{k-1}\frac{\bbC_{j,k}}{M_j},
\end{align*}
where $\bbV_k:=\var(\Psi_k-\Psi_{k-1})$ and $\bbC_{j,k}:=\cov(\Psi_j-\Psi_{j-1},\Psi_k-\Psi_{k-1})$.
Hence, the coupled estimator introduces a higher RMSE if the corrections $\Psi_\ell-\Psi_{\ell-1}$ are positively correlated across the levels.
In this case, we trade in variance for simulation time and the ratio of this trade-off is problem-dependent and hard to assess in advance. 

\section{Numerical Results}\label{astein_sec:num}
For our numerical experiment we consider $\bbD=(0,1)^2$ with $T=1$, initial data $u_0(x_1,x_2)=\frac{1}{10}\sin(\pi x_1)\sin(\pi x_2)$, source term $f\equiv1$ and set $\bar a\equiv0$. 
The covariance operator $Q$ of $W$ is given by the by the \emph{Mat\'ern covariance function}
\begin{equation*}
[Q\varphi](y):=\int_\bbD \sigma^2\frac{2^{1-\nu}}{\Gamma(\nu)}\Big(\sqrt{2\nu}\frac{\|x-y\|_2}{\chi}\Big)^\nu K_{\nu}\Big(\sqrt{2\nu}\frac{\|x-y\|_2}{\chi}\Big)\varphi(x)dx,\quad \varphi\in H,
\end{equation*}
with smoothness parameter $\nu>0$, variance $\sigma^2>0$ and correlation length $\chi>0$. 
Above, $\Gamma$ denotes the Gamma function, $\|\cdot\|_2$ is the Euclidean norm in $\R^2$ and $K_\nu$ is the modified Bessel function of the second kind with $\nu$ degrees of freedom.
We set the covariance parameters as $\nu=1.5,\sigma=0.5$ and $\chi=0.1$, hence Assumption~\ref{astein_ass:EV} is fulfilled, see \cite{Gr15}.
To approximate the Gaussian field, we use the circulant embedding method from \cite{Gr18a} to draw samples of $W$ at a grid of discrete points in $\bbD$ and then use linear interpolation to obtain an extension to $\overline\bbD$. We choose a maximum distance of $\epsilon>0$ for the grid points and denote the corresponding approximation by $W_\epsilon$.
Furthermore, we set $\Phi(\cdot)=\exp(\cdot)$ and observe that for any $s\in[1,\infty)$ 
\begin{align*}
\|\Phi(W)-\Phi(W_\epsilon)\|_{L^s(\Omega;L^\infty(\bbD))}
&\le C\EE\Big(\big(\sum_{j=1}^d\|\partial_{x_j}\Phi(W)\|_{L^\infty( \bbD)}\epsilon\big)^s\Big)^{1/s}\le C\epsilon
\end{align*}
holds by the path-wise Lipschitz regularity of $W$ and Lemma~\ref{astein_lem:a} (cf. Assumption~\ref{astein_ass:ahat}).

For the discontinuous random field $P$, we denote by $\calU((c_1,c_2))$ the uniform distribution on the interval $(c_1,c_2)\subset\R$,
sample four i.i.d. $\calU((0.2,0.8))$-distributed random variables $U_1,\dots,U_4$ and assign one $U_i$ to each side of the square $\partial\bbD$.
We then connect the points on two opposing edges by a straight line to obtain a random partition $\calT$ consisting of $\tau=4$ convex quadrangles.
Finally, we assign independent jump heights $P_1,P_2\sim\calU((0,1)), P_3\sim\calU((5,6))$ and $P_4\sim\calU((10,11))$ to the partition elements, such that two adjacent elements do not have the same jump distribution. This guarantees rather steep discontinuities across the interfaces in $\calT$, see Figure~\ref{astein_fig:sample}.
We do not need any approximation procedure for $P$ and obtain 
$a_\epsilon:=\exp(W_\epsilon) + P$.
Clearly, $a_\epsilon$ satisfies Assumption~\ref{astein_ass:ahat} and we define $b_\epsilon:=\max(-2a_\epsilon,-5)$.
The QoI is given by 
\begin{equation*}
\Psi(u):=\int_\bbD u(x)\exp(-0.25\|(0.25,0.75)-x\|_2^2)dx.
\end{equation*}

For the sample-adapted FE approach, we set the refinement parameters to $\overline h^{(a)}_\ell=\frac{1}{4}2^{-\ell/2}$ for $\ell\in\N_0$ and choose $\epsilon^{(a)}_\ell=\Delta t^{(a)}_\ell=(\overline h^{(a)}_\ell)^2$. While this choice gives an error equilibrium for $\kappa=1$, it ensures that for any $\kappa<1$ the RMSE is dominated solely by the spatial discretization error. Thus, we may infer the true value of $\kappa$ from the numerical experiment.
We also consider a non-adapted FE method with fixed and deterministic triangulations on $\bbD$.
For given approximation parameters $\epsilon,\overline h^{(na)}_\ell$ and $\Delta t$ in the non-adapted setting, we may not expect a better error bound than 
\begin{equation*}
\EE(|\Psi-\Psi_{\epsilon,\ell,\Delta t}|^2)^{1/2}\le C(\epsilon+\overline h^{(na)}_\ell+\Delta t)
\end{equation*}
in Corollary~\ref{astein_cor:overall_error}.
This is due to the fact that the standard FE method for elliptic problems with discontinuous coefficients does not converge at a better rate than $\calO((\overline h^{(na)})^{1/2})$ in the $V$-norm, see \cite[Remark 4.2]{BS18b}. Thus, if we consider again the dual problem as in Theorem~\ref{astein_thm:semi_error_H}, we may not expect a better rate than $\calO(\overline h^{(na)})$ with respect to the $H$-norm.
We choose the non-adapted FE grid with diameter $\overline h^{(na)}_\ell:=\frac{1}{4}2^{-\ell}$ and set accordingly $\epsilon^{(na)}_\ell=\Delta t^{(na)}_\ell =\overline h^{(na)}_\ell$.
In both FE methods, we use the midpoint rule on each triangle to approximate the entries of the stiffness matrix. The resulting quadrature error is of order $\calO(\overline h_\ell^2)$ with respect to the $H$-norm in the sample-adapted case and hence does not dominate the overall approximation error, see \cite[Section 2]{Gr18b}. For non-adapted FE, no a-priori estimate on the quadrature error is possible due to the discontinuities in $a$ and $b$, but our results suggest that this bias also in line with the overall approximation error. 
As $\epsilon_{\ell-1}=2\epsilon_\ell$, the circulant embedding grids (to sample $W_\epsilon$) are nested and we may achieve the MLMC coupling by first generating the discrete set of points on level $\ell$ and then taking the appropriate subset of points for level $\ell-1$.  

In the sample-adapted MLMC algorithm, we choose the number of samples via
\begin{equation*}
(M^{(a)}_0)^{-1}=\ceil{(\overline h^{(a)}_L)^4}\quad\text{and}\quad (M^{(a)}_\ell)^{-1}=\ceil{\frac{1}{4}\frac{(\overline h^{(a)}_L)^4}{(\overline h^{(a)}_\ell)^4}\Big(\frac{(\ell+1)^{-1.001}}{\sum_{k=1}^L(k+1)^{-1.001}}\Big)^{-2}}
\end{equation*}
for $\ell=1,\dots,L$, whereas, we choose 
\begin{equation*}
(M^{(na)}_0)^{-1}=\ceil{(\overline h^{(na)}_L)^2}\quad\text{and}\quad (M^{(na)}_\ell)^{-1}=\ceil{\frac{(\overline h^{(na)}_L)^2}{(\overline h^{(na)}_\ell)^2}\Big(\frac{(\ell+1)^{-1.001}}{\sum_{k=1}^L(k+1)^{-1.001}}\Big)^{-2}}
\end{equation*}
in the non-adapted MLMC approach.
Basically, we choose $1/M_\ell$ proportional to $\mathbb V_\ell=\text{Var}(\Psi_\ell-\Psi_{\ell-1})$ on each level and thus distribute the errors equally across all levels. Another possibility would be to distribute the computational effort equally (see \cite{G15}), which requires estimates on the cost of a single sample on each level.
The sequence $(\ell^{-c},\ell\in\bbN)$ decreases rapidly for $c>1$ and sums up to $\zeta(c)<+\infty$, where $\zeta(\cdot)$ is the Riemann $\zeta$-function. Hence, the above choice of $\rho_i$ ensures that only a few expensive samples on high levels are necessary and, due to the uniform bound $\sum_{\ell=1}^L\rho_\ell<\zeta(c)$, it is well suited to compare estimators for a varying choice of $L$. 
In terms of Theorem~\ref{astein_thm:mlmc}, we have chosen $C_\rho=2$ for the number of samples in the sample-adapted method, whereas $C_\rho=1$ for standard FE.
Similar calculations as in Theorem~\ref{astein_thm:mlmc} show that this choice leads  to $\|\Psi-E^L(\Psi_L)\|_{L^2(\Omega,\R)}\le C(2^{-2-L})$ in either case, where the constant $C$ is the same for adapted and non-adapted FE. Hence, $C_\rho$ is merely a normalizing constant and the above choice of $M_\ell$ ensures that both approaches produce a comparable error for fixed $L$. 
Finally, we calculate a reference QoI $\Psi_{ref}:=E^L(\Psi_L)$ with $L=7$ and the sample-adapted method and estimate the relative RMSE 
$\|\Psi_{ref}-E^L(\Psi_L)\|_{L^2(\Omega,\R)}/\Psi_{ref}$ for $L=0,\dots,5$ based on 50 independent samples of $E^L(\Psi_L)$ for the sample-adapted and non-adapted MLMC algorithm. 
For each approach, we use adapted/non-adapted FE combined with a standard/coupled MLMC estimator, thus we compare a total of four algorithms regarding their error decay and efficiency.

\begin{figure}[ht]
	\centering
	\subfigure{\includegraphics[height=0.2\textheight, width=0.49\textwidth,]{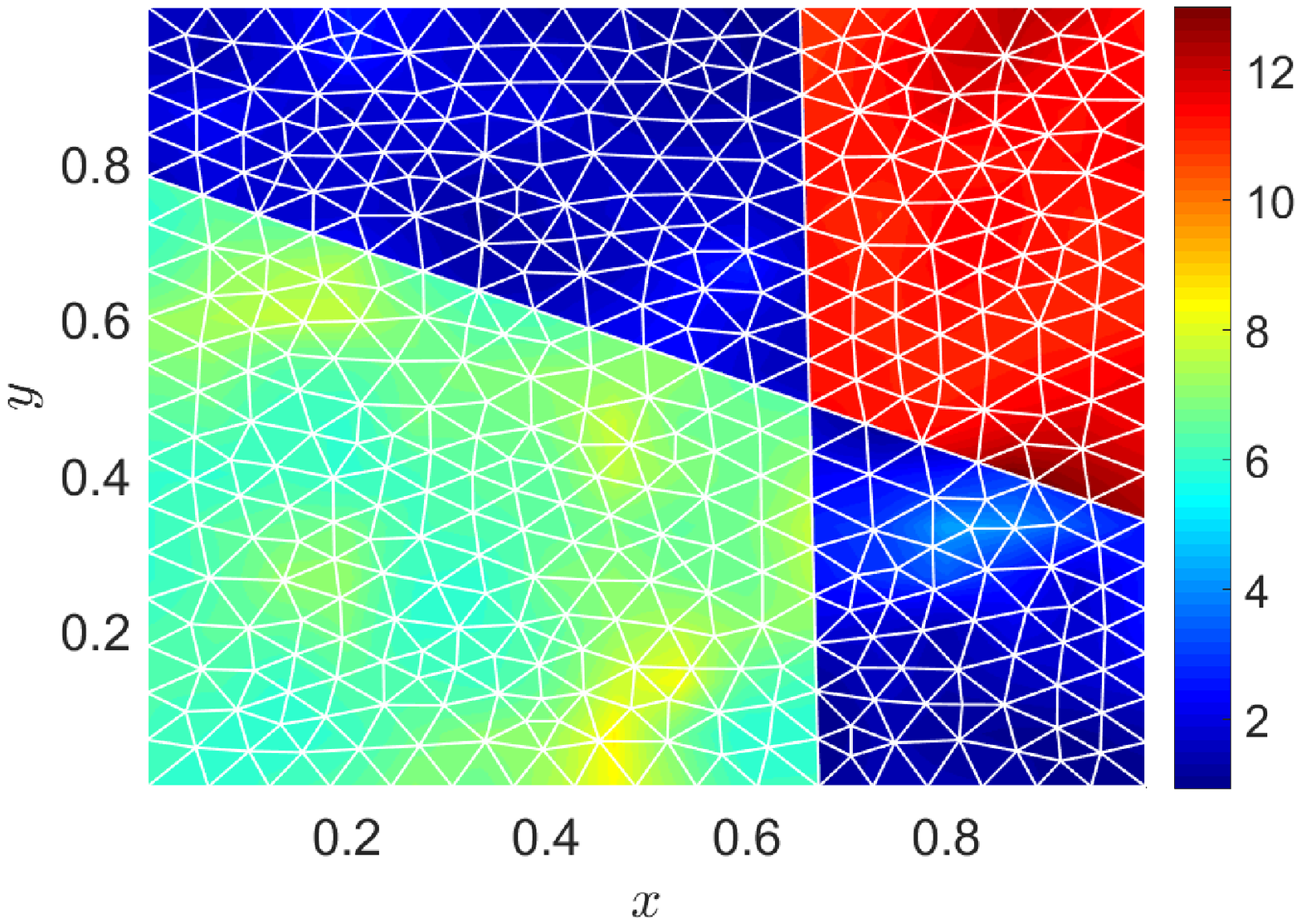}}
	\subfigure{\includegraphics[height=0.2\textheight, width=0.49\textwidth,]{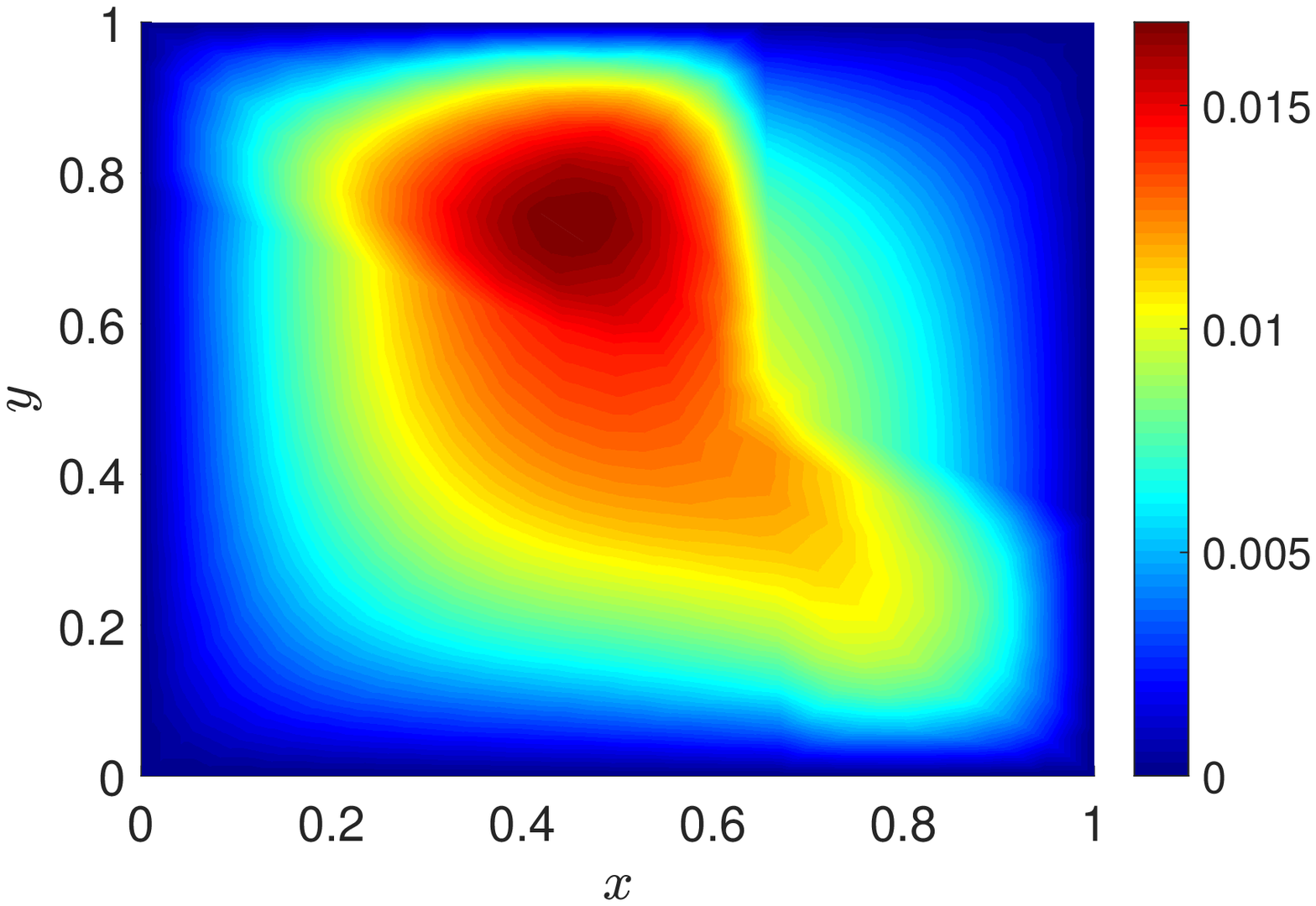}}
	\subfigure{\includegraphics[height=0.2\textheight, width=0.49\textwidth,]{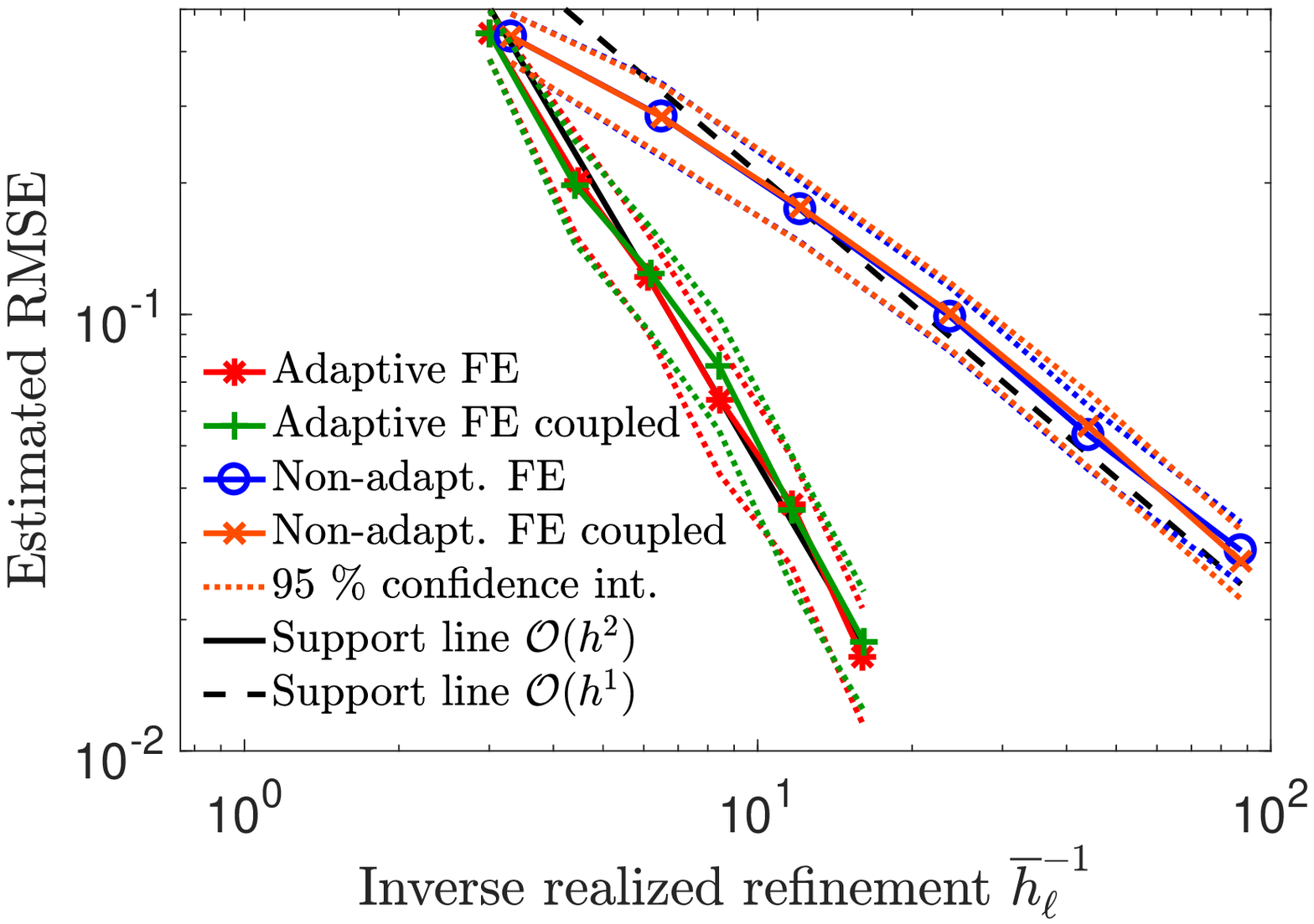}}
	\subfigure{\includegraphics[height=0.2\textheight, width=0.49\textwidth,]{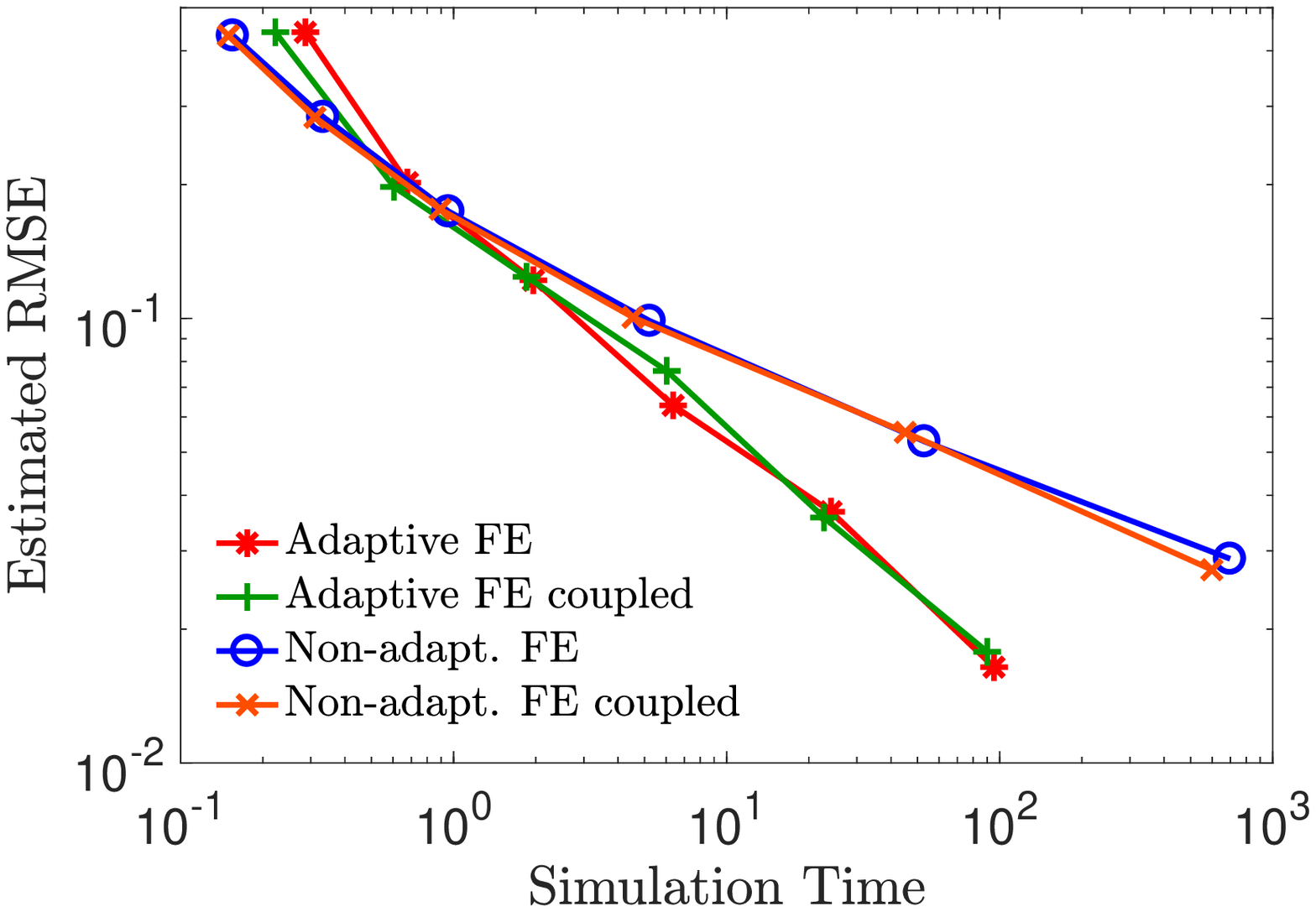}}
	\caption{Top: Sample of the diffusion coefficient with sample-adapted FE grid (left) and FE solution at $T=1$ (right). Bottom: RMSE vs. refinement (left) and RMSE vs. simulation time (right).}
	\label{astein_fig:sample}
\end{figure}

Figure~\ref{astein_fig:sample} confirms our theoretical results from Section~\ref{astein_sec:para_disc}, i.e. the sample-adapted spatial discretization yields rate $\calO(\overline h^2_\ell)$ compared to $\calO(\overline h_\ell)$ in the non-adapted setting. Hence, we are able to choose coarser spatial grids in the first approach which entails a better time-to-error ratio for both sample-adapted methods.
The results also indicate that $\kappa\approx 1$ holds for this particular example, otherwise we would see a lower rate of convergence than $\calO(\overline h^2_\ell)$ for the sample-adapted methods. 
While the sample-adapted FE grids have to be generated new for each sample, the $L+1$ deterministic grids for the non-adapted FE method are generated and stored before the Monte Carlo loop. However, as we see from the time-to-error plot, the extra work of renewing the FE meshes for each sample in the sample-adapted method is more than compensated by the increased order of convergence. 
The computational cost of the sample-adapted MLMC estimators are (roughly) inversely proportional to the squared errors, which is the best possible results one may achieve with MLMC, see \cite{G15} and the references therein.
To conclude, we remark that the coupled MLMC estimator yields a slight gain in efficiency if combined with non-adapted FE, whereas it produces similar results when using the sample-adapted discretization. We emphasize that there are scenarios where the coupled estimator outperforms standard MLMC and, on the other hand, there are examples were coupling performs worse due to high correlation terms $\bbC_{j,k}$ (for both, we refer to numerical examples in \cite{BS18b}.) Hence, even though performance is similar to standard MLMC, it makes sense to consider the coupled estimator in our scenario. As we have mentioned at the end of Section~\ref{astein_sec:mlmc}, this behavior may not be expected a-priori. 

\begin{acknowledgement}
The research leading to these results has received funding from the German Research Foundation (DFG) as part of the Cluster of Excellence in Simulation Technology (EXC 310/2) at the University of Stuttgart and it is gratefully acknowledged.
\end{acknowledgement}


\end{document}